\documentclass{amsart}
\usepackage{graphicx}
\usepackage{amsmath,amssymb}
\usepackage{amsmath}

\newtheorem{thm}{Theorem}[section]
\newtheorem{cor}[thm]{Corollary}

\newtheorem{prop}[thm]{Proposition}
\theoremstyle{definition}
\newtheorem{defn}[thm]{Definition}
\theoremstyle{remark}

\numberwithin{equation}{section}

\begin{document}
\title[On the Operator-valued $\mu$-cosine]{On the Operator-valued $\mu$-cosine functions }
\author[B. Bouikhalene and  E. Elqorachi]{ Bouikhalene Belaid and  Elqorachi Elhoucien}
\thanks{Keywords: cosine equation; locally compact group; unitary representation; character, multiplicative function.}
\thanks{ 2010 MSC: 47D09; 22D10; 39B42}

\begin{abstract} Let $(G,+)$ be a topological abelian group with a neutral element $e$ and
 let  $\mu : G\longrightarrow\mathbb{C}$ be a continuous character of $G$.
 Let $(\mathcal{H}, \langle \cdot,\cdot \rangle)$ be a complex Hilbert space and
 let $\mathbf{B}(\mathcal{H})$ be the algebra of all linear continuous operators
 of $\mathcal{H}$ into itself. A continuous mapping $ \Phi: G\longrightarrow \mathbf{B}(\mathcal{H})$
  will be called an operator-valued $\mu$-cosine function if it satisfies both the $\mu$-cosine equation $$\Phi(x+y)+\mu(y)\Phi(x-y)=2\Phi(x)\Phi(y),\; x,y\in G$$
and the condition $\Phi(e)=I,$ where $I$ is the identity of
$\mathbf{B}(\mathcal{H})$. We show that any hermitian
operator-valued $\mu$-cosine functions has the form
$$\Phi(x)=\frac{\Gamma(x)+\mu(x)\Gamma(-x)}{2}$$ where $ \Gamma: G\longrightarrow \mathbf{B}(\mathcal{H})$
is a continuous  multiplicative operator. As an application, positive definite kernel theory and
W. Chojnacki's results on the uniformly bounded normal cosine operator are used to give explicit formula of solution of the cosine equation.
\end{abstract}
\maketitle

\section{Introduction}
\subsection{}
The $\mu$-cosine equation, also called the pre-d'Alembert equation,
on abelian group $G$ is the equation
\begin{equation}f(x+y)+\mu(y)f(x-y)=2f(x)f(y),\; x,y\in G\end{equation}
where $f: G\longrightarrow \mathbb{C}$ is the unknoun. Davison [8]
gave solution of (1.1) in terms of traces of certain representations
of $G$ on $\mathbb{C}^{2}$. In [22] Stetk{\ae}r proves that a
non-zero solution of (1.1) has the form
\begin{equation}f(x)=\frac{\chi(x)+\mu(x)\chi(-x)}{2}, x\in
G,\end{equation} where $\chi$ is a character of $G$. In the case
where $\mu=1$, equation (1.1) becomes the classic cosine functional
equation (also called the d'Alembert functional equation)
\begin{equation} f(x+y)+f(x-y)=2f(x)f(y), \; x, y\in G.
\end{equation}
Several mathematicians studied the equation (1.3). The monographs by
Acz\'{e}l [2] and by Acz\'{e}l and Dhombres [3] have references and
detailed discussions. The main purpose of this work is to extend
equation (1.1) to functions taking values in the algebra
$\mathbf{B}(\mathcal{H})$ of bounded operators on a Hilbert space
$\mathcal{H}$.
\subsection{}
Throughout this paper, $G$ will be a topological abelian group with
the unit element $e$. The space of continuous complex-valued
functions is denoted by $\mathcal{C}(G)$ and the set of all
continuous homomorphisms $\gamma : G\longrightarrow
\mathbb{C}\setminus\{0\}$ by $\mathcal{M}(G)$. Let $\mu :
G\longrightarrow\mathbb{C}^{*}$ be a continuous character of the
group $G$ i.e. $\mu\in \mathcal{M}(G)$ such that $\mu(e)=1$. For all
$f\in \mathcal{C}(G)$ we define the function $\check{f}$ by
$\check{f}(x)=f(-x)$. Let $(\mathcal{H}, \langle, \rangle)$ be a
Hilbert space over $\mathbb{C}$ and let $\mathbf{B}(\mathcal{H})$ be
the algebra of all linear continuous operators of $\mathcal{H}$ into
itself with the usual operator norm denoted $\|.\|$. A mapping $
\Phi: G\longrightarrow \mathbf{B}(\mathcal{H})$ is said to be
hermitian if it satisfies $\Phi^{*}(x)=\Phi(-x)$ for all $x\in G$,
where $\Phi^{*}(x)$ is the adjoint operator of $\Phi(x)$. A
continuous mapping $ \Gamma: G\longrightarrow
\mathbf{B}(\mathcal{H})$ is said to be a multiplicative operator if
$\Gamma(x+y)=\Gamma(x)\Gamma(y)$ for all $x,y\in G$ and
$\Gamma(e)=I$. Also  we say that a continuous mapping $ \Phi :
G\longrightarrow \mathbf{B}(\mathcal{H})$ is an operator
 valued $\mu$-cosine function if  it satisfies both the $\mu$-cosine functional equation
\begin{equation} \Phi(x+y)+\mu(y)\Phi(x-y)=2\Phi(x)\Phi(y), \; x, y\in G
\end{equation}
and the conditions  $\Phi(e)=1$. The scalar case of (1.4) is given
by the equation (1.1).  For $\mu=1$ we obtain the cosine functional
equation
\begin{equation} \Phi(x+y)+\Phi(x-y)=2\Phi(x)\Phi(y), \; x, y\in G.
\end{equation}
Several variants of (1.5) has been studied by Kisy\'{n}ski [10] and
[11], Sz\'{e}kelyhidi [23], Chojnacki [6] and [7], Stetk{\ae}r [19],
[20] and [21].
\subsection{}
The main purpose of this work is to solve  the  equation (1.4),
where the unknown $\Phi$ is an hermitian continuous functions on $G$
taking its values in $\mathbf{B}(\mathcal{H})$ or in the algebra
$M_{n}(\mathbb{C})$ of complex $n\times n$ matrices. By using
positive definite kernels and linear algebra theory we find that any
hermitian continuous solution of (1.4) has the form
$\Phi(x)=\frac{\Gamma(x)+\mu(x)\Gamma(-x)}{2}$ where $ \Gamma:
G\longrightarrow \mathbf{B}(\mathcal{H})$ is a continuous
multiplicative operator.
\subsection{Notation and preliminary   }
\begin{defn}
A continuous function $K : G\times G\longrightarrow\mathbb{C}$ is
said to be a positive definite kernel on $G$ if for all $n\in
\mathbb{N}$, $x_{1},...,x_{n}\in G$ and arbitrary complex numbers
$c_{1},c_{2},...,c_{n}$ we have
\begin{equation}\sum_{i=1}^{n}\sum_{j=1}^{n}c_{i}\overline{c_{j}}K(x_{i},x_{j})\geq 0.
\end{equation}
\end{defn}
We provide some known results on positive definite kernel theory.
For more details we refer to [15].
\begin{prop} Let $K$ be a positive definite kernel on $G$ and let $$V_{K}:=span\{K(x,\cdot) : x\in G\}.$$ Then\\
i) $V_{K}\subset \mathcal{C}(G)$,\\
ii) $V_{K}$ is equipped with the inner product
$$\langle f,g\rangle_{K}=\sum_{i=1}^{n}\sum_{j=1}^{n}\alpha_{i}\overline{\beta_{j}}K(x_{i},x_{j})$$
where\\ $f=\sum_{i=1}^{n}\alpha_{i}K(x_{i},.)$,
$g=\sum_{i=1}^{m}\beta_{j}K(x_{j},.)$,\\ $x_{1},...,x_{sup(m,n)}\in
G$
and $\alpha_{1},...,\alpha_{n},\beta_{1},...,\beta_{m}\in \mathbb{C}$.\\
Let $\mathcal{H}_{K}$ be the completion of $V_{K}$. Then
($\mathcal{H}_{K}$,$\langle . ,.\rangle_{K}$) is a Hilbert space of
continuous functions on $G$. The function $K$ is the reproducing
kernel of the Hilbert space $\mathcal{H}_{K}$.
\end{prop}
\begin{thm}
Let $K$ be a positive definite kernel on $G$. Then there exists a
Hilbert space ($\mathcal{H}_{K}$,$\langle . ,.\rangle_{K}$) and a
continuous mapping $$T: G\longrightarrow \mathcal{H}_{K},
x\longmapsto K(x,\cdot),$$
such that \\
1) $K(x,y)=\langle T(x),T(y)\rangle_{K}$ for all $x,y\in G$.\\
2) $span\{T(x) : x\in G\}$ is dense in $\mathcal{H}_{K}$.\\
Moreover, the pair $(\mathcal{H}_{K},T)$ is unique in the following
way : if another pair ($\mathcal{L},U$) satisfies (1) and (2), there
exists a unique unitary isomorphism $\Psi:
\mathcal{H}_{K}\longrightarrow \mathcal{L}$ such that $U=\Psi\circ
T$.\end{thm} For all $f\in \mathcal{C}(G)$ and for all $x,y\in G$ we
define
\begin{equation}K_{f}(x,y):=\frac{1}{2}\{f(-y+x)+\mu(x)f(-y-x)\}.\end{equation}
and
\begin{equation}
f(x)=f^{+}_{\mu}(x)+f^{-}_{\mu}(x), \; x\in G,
\end{equation}
where $f_{\mu}^{+}(x)=\frac{f(x)+\mu(x)f(-x)}{2}$ and
$f_{\mu}^{-}(x)=\frac{f(x)-\mu(x)f(-x)}{2}$
\section{ General properties }
\begin{prop}
Let $ \Phi: G\longrightarrow \mathbf{B}(\mathcal{H})$  be a solution of (1.4). Then\\
i) $\Phi(-x)=\mu(-x)\Phi(x)$ for all $x\in G$.\\
ii) $\Phi(x)\Phi(y)=\Phi(y)\Phi(x)$ for all $x, y\in G$.\\
iii) For all invertible operator $S\in \mathbf{B}(\mathcal{H})$ we
have $S\Phi(x)S^{-1}$ for all $x\in G$ is a solution of (1.4).
\end{prop}
\begin{proof}
i) For all $x,y \in G$ we have
\begin{eqnarray*}2\Phi(x)\Phi(y)&=&\Phi(x+y)+\mu(y)\Phi(x-y)\\
&=&\mu(y)(\Phi(x-y)+\mu(-y)\Phi(x+y))\\&=&
2\mu(y)\Phi(x)\Phi(-y).\end{eqnarray*} From which we get that
\begin{equation}\Phi(x)\Phi(y)=\mu(y)\Phi(x)\Phi(-y).\end{equation}
Setting $x=e$ in (2.1) and $\Phi(e)=I$ we get that $\Phi(y)=\mu(y)\Phi(-y)$ for all $y\in G$.\\
ii) For all $x,y\in G$ we have
\begin{eqnarray*}2\Phi(y)\Phi(x)&=&\Phi(y+x)+\mu(x)\Phi(y-x)\\&=&\Phi(x+y)+\mu(x)\mu(y-x)\Phi(x-y)\\&=&
\Phi(x+y)+\mu(y)\Phi(x-y)\\&=&
2\Phi(x)\Phi(y).\end{eqnarray*} From which we get that $\Phi(x)\Phi(y)=\Phi(y)\Phi(x)$ for all $x, y\in G$.\\
iii) For all $x,y \in G$ we have
\begin{eqnarray*}S\Phi(x+y)S^{-1}+\mu(y)S\Phi(x-y)S^{-1}&=&S(\Phi(x+y)+\mu(y)\Phi(x-y))S^{-1}\\
&=&S2\Phi(x)\Phi(y)S^{-1}\\
&=&2S\Phi(x)S^{-1}S\Phi(y)S^{-1}.\end{eqnarray*} From which we get
that
$$S\Phi(x+y)S^{-1}+\mu(y)S\Phi(x-y)S^{-1}=2S\Phi(x)S^{-1}S\Phi(y)S^{-1}$$ for all $x, y\in G$. Furthermore we have
$$S\Phi(e)S^{-1}=I.$$
\end{proof}
\begin{prop}
Let $M: G\longrightarrow \mathbf{B}(\mathcal{H})$ be a multiplicative operator. Then \\
$$\Phi(x)=\frac{M(x)+\mu(x)M(-x)}{2}, \;x\in G$$
is an operator-valued $\mu$-cosine functions.
\end{prop}
\begin{proof}
Since $M(x+y)=M(x)M(y)$ for all $x,y \in G$ and $M(e)=I$, we get by
easy computations that $\Phi$  is an operator-valued $\mu$-cosine
functions.
\end{proof}
By easy computations we get the following proposition
\begin{prop} For all $f\in \mathcal{C}(G)$ we have the following statements\\
i) the mapping $(x,y)\longmapsto K_{f}(x,y)$ is continuous.\\
ii) $K_{f}(-x,y)=\mu(-x)K_{f}(x,y)$ for all $x,y\in G$.\\
iii) $K_{f}(0,-y)=f(y)$ for all $y\in G$.\\
4i) $K_{f}(x,0)=f^{+}_{\mu}(x)$ for all $x\in G$.
\end{prop}
We need the following proposition in the main result
\begin{prop} Let $ \Phi: G\longrightarrow \mathbf{B}(\mathcal{H})$ be be a solution of (1.4) such that $\Phi(x)^{*}=\Phi(-x)$ for all $x\in G$ and let
$f: G \longrightarrow \mathbb{C}$, $x\longmapsto \langle \Phi(x)\xi,\xi\rangle $ for  $\xi\in \mathcal{H}$. Then\\
i) $f(-x)=\mu(-x)f(x)$ for all $x\in G$.\\
ii) $\overline{f(x)}=f(-x)$ for all $x\in G$.\\
iii) $K_{f}(x,y)=\langle \Phi(x)\xi,\Phi(y)\xi\rangle$ for all $x,y\in G$.\\
4i) $K_{f}$ is a positive definite kernel.\\
 5i) $K_{f}(x,.)=\frac{1}{2}\{(R_{-x}\check{f})+\mu(x)(R_{x}\check{f})\}$ for all $x\in G$ where $R$ is the right regular representation of $G$.
\end{prop}
\begin{proof}
i) Since $\Phi(-x)=\mu(-x)\Phi(x)$ for all $x\in G$ we get that
$$f(-x)=\langle\Phi(-x)\xi,\xi\rangle=\langle\mu(-x)\Phi(x)\xi,\xi\rangle=\mu(-x)\langle\Phi(x)\xi,\xi\rangle=\mu(-x)f(x).$$
ii) for all $x\in G$ we have
$$\overline{f(x)}=\overline{\langle\Phi(-x)\xi,\xi\rangle}=\langle\xi,\Phi(x)\xi\rangle
=\langle\Phi(x)^{*}\xi,\xi\rangle=\langle\Phi(-x)\xi,\xi\rangle=f(-x).$$
iii) For all $x,y\in G$ we have
\begin{eqnarray*}K_{f}(x,y)&=&\frac{1}{2}\{f(-y+x)+\mu(x)f(-y-x)\}\\
&=&\frac{1}{2}\{(\langle\Phi(-y+x)+\mu(x)\Phi(-y+x)\xi,\xi\rangle)\}\\&=&
\langle\Phi(-y)\Phi(x)\xi,\xi\rangle\\&=&
\langle\Phi(y)^{*}\Phi(x)\xi,\xi\rangle\\&=&
\langle\Phi(x)\xi,\Phi(y)\xi\rangle\end{eqnarray*} 4i) For all $n\in
\mathbb{N}$, $x_{1},...,x_{n}\in G$ and arbitrary complex numbers
$c_{1},c_{2},...,c_{n}$ we have
\begin{eqnarray*}\sum_{i=1}^{n}\sum_{j=1}^{n}c_{i}\overline{c_{j}}K_{f}(x_{i},x_{j})&=&
\sum_{i=1}^{n}\sum_{j=1}^{n}c_{i}\overline{c_{j}}\langle\Phi(x_{i})\xi,\Phi(x_{j})\xi\rangle\\
&=&\langle\sum_{i=1}^{n}c_{i}\Phi(x_{i})\xi,\sum_{j=1}^{n}c_{j}\Phi(x_{j})\xi\rangle\\
&=&\|\sum_{i=1}^{n}c_{i}\Phi(x_{i})\|^{2}\geq 0.\end{eqnarray*} 5i)
For all $x,y\in G$ we have
\begin{eqnarray*}K_{f}(x,y)&=&\frac{1}{2}\{f(-y+x)+\mu(x)f(-y-x)\}\\
&=& \frac{1}{2}\{(R_{x}f)(-y)+\mu(x)(R_{-x}f)(-y)\}\\
&=&\frac{1}{2}\{\check{(R_{x}f)}(y)+\mu(x)\check{(R_{-x}f)}(y).\}
\end{eqnarray*} Since $\check{(R_{x}f)}=R_{-x}\check{f}$ for all $x\in G$ it follows that $K_{f}(x,y)=\frac{1}{2}\{R_{-x}\check{f}(y)+\mu(x)(R_{-}\check{f})(y)\}$ for all $x,y\in G$. So that we have
$K_{f}(x,.)=\frac{1}{2}\{(R_{-x}\check{f})+\mu(x)(R_{x}\check{f})\}$
for all $x\in G$.
\end{proof}
\section{Main Result}
In the next theorem we solve the equation (1.4).
\begin{thm}
Let $ \Phi: G\longrightarrow \mathbf{B}(\mathcal{H})$ be an
hermitian operator-valued $\mu$-cosine functions. Then there exists
a multiplicative operator $M : G\longrightarrow
\mathbf{B}(\mathcal{H})$ such that
$$\Phi(x)=\frac{M(x)+\mu(x)M(-x)}{2}, \; x\in G.$$
\end{thm}
\begin{proof}
Let $\xi\in \mathcal{H}$. By the same way as in the proof of Theorem 2.2 in [1], we can suppose that the vector $\xi$ is cyclic\\
Let now  $\varphi(x)=\langle \Phi(x)\xi,\xi\rangle$ for all $x\in
G$. For all $x,y\in G$ we have
$K_{\varphi}(x,y)=\langle\Phi(x)\xi,\Phi(y)\xi\rangle$. So that
$K_{\varphi}$ is a positive definite kernel. According to Theorem
1.3 there exists a Hilbert space
$(\mathcal{H}_{\varphi},\langle\cdot,\cdot \rangle_{\varphi})$ and a
mapping $T: G\longrightarrow\mathcal{H}_{\varphi}$, $x\longmapsto
K_{\varphi}(\cdot,x)$ such that
\begin{eqnarray*}K_{\varphi}(x,y)&=&\langle K_{\varphi}(x,\cdot),K_{\varphi}(y,\cdot)\rangle\\&=&
\langle\Phi(x)\xi,\Phi(y)\xi\rangle\end{eqnarray*} and a unique
unitary isomorphism  $\psi : \mathcal{H}_{\varphi}\longrightarrow
\mathcal{H}$ such tat
\begin{eqnarray}\Phi(x)\xi=\psi(K_{\varphi}(x,.))=\frac{1}{2}\psi[\check{(R_{x}\varphi)}+\mu(x)\check{(R_{-x}\varphi)}].\end{eqnarray}
Since $\Phi(e)=I$ we get by setting $x=e$ in (3.1)  that $\xi=\psi(\check{\varphi})$. From which we get that $\check{\varphi}=\psi^{-1}(\xi)$. We show that $\check{(R_{x}\varphi)}=R_{-x}\check{\varphi}$ and $\check{(R_{-x}\varphi)}=R_{x}\check{\varphi}$ for all $x\in G$.\\
So that for all $x\in G$ and $\xi\in \mathcal{H}$ we have
$$\Phi(x)\xi=\frac{1}{2}\psi[(R_-x+\mu(x)R_x)\psi^{-1}(\xi)].$$
Hence $\Phi(x)=\psi\circ R(x)\circ \psi^{-1}$ for all $x\in G$.\\
Since $R(x)=\frac{R_{-x}+\mu(x)R_{x}}{2}$ for all $x\in G$ we get
that
$$\Phi(x)=\frac{\psi\circ R_{-x}\circ \psi^{-1}+\mu(x)\psi \circ R_{x}\circ \psi^{-1}}{2}.$$
Setting $M(x)=\psi\circ R_{-x}\circ \psi^{-1}$ for all $x\in G$. We
have for all $x,y \in G$ that
\begin{eqnarray*}M(x+y)&=&\psi\circ R_{-x-y}\circ \psi^{-1}\\
&=&\psi\circ R_{-x}\circ R_{-y}\circ \psi^{-1}\\
&=&\psi\circ R_{-x}\circ \psi\circ \psi^{-1}\circ R_{-y}\circ
\psi^{-1}\\&=& M(x)M(y).
\end{eqnarray*}
and that $M(e)=\psi\circ R_{e}\circ \psi^{-1}=I$. \\
Finally we have that $\Phi(x)=\frac{M(x)+\mu(x)M(-x)}{2}$ for all
$x\in G$ where $M: G\longrightarrow \mathbf{\mathcal{H}}$ is a
multiplicative operator. This ends the proof of theorem.
\end{proof}
In the next corollary we determine solutions of (1.4) taking their
values in the complex $n\times n$ matrices
\begin{cor}
Let $ \Phi: G\longrightarrow M_{n}(\mathbb{C})$ be a continuous
hermitian solution of (1.4). Then there exists $A\in
GL(n,\mathbb{C})$ such that
\begin{equation}\Phi(x)=A\frac{E(x)+\mu(x)E(-x)}{2}A^{-1}, \; x\in G\end{equation}
where $E : G\longrightarrow M_{n}(\mathbb{C})$ has the form
$$\begin{pmatrix}
\gamma_{1} & 0&...&0 \\
0 &\gamma_{2}& ...&0\\
0&0& ...&0\\
...&...& \gamma_{i}&...\\
0&0& ...&\gamma_{n}
\end{pmatrix}
$$
where $\gamma_{1},...,\gamma_{n}\in M(G)$ and $\gamma_{i}\neq
\gamma_{j}$ for all $i,j\in \{1,...,n\}$ such that $i\neq j$
\end{cor}
\begin{proof}
since $\Phi(x)\Phi(y)=\Phi(y)\Phi(x)$ for all $x,y\in G$ and
$\Phi(x)^{*}=\Phi(-x)$ for all $x\in G$ it follows that $\Phi(x)$
for all $x\in G$ can be diagonalized simultaneously. So there exists
$ A\in GL(n,\mathbb{C})$ such that
$$\Phi(x)=A\begin{pmatrix}
\omega_{1} & 0&...&0 \\
0 &\omega_{2}& ...&0\\
0&0& ...&0\\
...&...& \omega_{i}&...\\
0&0& ...&\omega_{n}
\end{pmatrix}A^{-1}.$$
Since $\Phi(x+y)+\mu(y)\Phi(x-y)=2\Phi(x)\Phi(y)$ for all $x,y\in G$
it follows that
$\omega_{i}(x+y)+\mu(y)\omega_{i}(x-y)=2\omega_{i}(x)\omega_{i}(y)$
for all $i\in \{1,...,n\}$. According to [22] there exists
$\gamma_{i}\in \mathcal{M}(G)$ for all $i\in \{1,...,n\}$. such that
$\omega(x)=\frac{\gamma_{i}(x)+\mu(x)\gamma_{i}(-x)}{2}$ for all
$x\in G$. So $\Phi(x)=A\frac{E(x)+\mu(x)E(-x)}{2}A^{-1}$ where
$E=\begin{pmatrix}
\gamma_{1} & 0&...&0 \\
0 &\gamma_{2}& ...&0\\
0&0& ...&0\\
...&...& \gamma_{i}&...\\
0&0& ...&\gamma_{n}
\end{pmatrix}$ and $\gamma_{i}\in \mathcal{M}(G)$ such that $\gamma_{i}\neq \gamma_{j}$ for $i\neq j$. This ends the proof of corollary
\end{proof}
\section{applications}
Throughout  this section we adhere to the terminology used in [7]. Let $G$ be a locally compact commutative group and let $\mu=1$. A mapping $\Phi: G\longrightarrow \mathbf{B}(\mathcal{H})$ will be said to be uniformly bounded if $ sup\{\|\Phi(x)\| : \; x\in G\}_{\mathbf{B}(\mathcal{H})}< +\infty $. The hermitian operator-valued cosine functions is denoted by $*$-operator-valued cosine functions in [7].\\\\
According to Theorem 1 in [7] we get the following proposition
\begin{prop}
Let $\Phi: G\longrightarrow \mathbf{B}(\mathcal{H})$ be a uniformly
bounded operator-valued cosine functions. Then there is an
invertible $S\in \mathbf{B}(\mathcal{H})$ such that
$\Psi(x)=S\Phi(x)S^{-1}$ for all $x\in G$ is an hermitian
operator-valued cosine functions
\end{prop}
In the next theorem we use our study to solve the equation (1.5)
\begin{thm}
Let $\Phi: G\longrightarrow \mathbf{B}(\mathcal{H})$ be a uniformly
bounded operator-valued cosine functions. Then there is an
invertible $S\in \mathbf{B}(\mathcal{H})$ and a multiplicative
operator $M : G\longrightarrow \mathbf{B}(\mathcal{H})$ such that
$$\Phi(x)=S\frac{M(x)+M(-x)}{2}S^{-1}, \; x\in G.$$
\end{thm}
\begin{proof}
By using Proposition 4.1 we get that $\Psi(x)=S\Phi(x)S^{-1}$  is a
solution of (1.5) such that $\Psi(x)^{*}=\Psi(x^{-1})$ for all $x\in
G$. According to Theorem 3.1 we get the remainder.
\end{proof}
 \vspace{1cm}
 Bouikhalene Belaid\\ Departement of Mathematics and Informatics\\
Polydisciplinary Faculty, Sultan Moulay Slimane university, Beni Mellal, Morocco.\\
E-mail : bbouikhalene@yahoo.fr.\\\\
 Elqorachi Elhoucien, \\Department of Mathematics,
\\Faculty of Sciences, Ibn Zohr University, Agadir,
Morocco,\\
E-mail: elqorachi@hotmail.com\\\\
\end{document}